\documentclass[12pt,a4paper,leqno,draft]{article}

\usepackage{amsmath, amsthm, amscd, amsfonts, amssymb, graphicx, color,makeidx}
\usepackage[cp1250]{inputenc}
\usepackage[T1]{fontenc} 

\usepackage{ifthen}
\usepackage{xstring}
\usepackage{authblk}
\usepackage{amsthm}
\usepackage{blindtext}
\usepackage{mathtools}

\newtheorem{theorem}{Theorem}[section]
\newtheorem{proposition}{Proposition}[section]
\newtheorem{lemma}{Lemma}[section]
\newtheorem{corollary}{Corollary}[section]
\theoremstyle{definition}
\newtheorem{definition}{Definition}[section]
\theoremstyle{remark}
\newtheorem{remark}{Remark}[section]
\newtheorem{example}{Example}[section]

\begin{document}

\title{On Drazin invertible $C^{*}$-operators and generalized $C^{*}$-Weyl operators}

\author{Stefan  Ivkovi\'{c}, stefan.iv10@outlook.com}
\affil{Mathematical Institute of the Serbian Academy of Sciences and Arts,  Kneza Mihaila 36,  11000 Beograd,Serbia}
\date{}

\maketitle

\abstract{Generalized Weyl operators on Hilbert spaces have been introduced and studied by Djordjevi\'{c} in \cite{DDj2}. In this paper, we provide a generalization of his result in the setting of $C^{*}$-operators on Hilbert $C^{*}$-modules by giving sufficient conditions under which the sum of a generalized $C^{*}$-Weyl operator and a finitely generated $C^{*}$-operator is a generalized $C^{*}$-Weyl operator. Also, we obtain an extension of Djordjevi\'{c}'s results from the case of operators on Hilbert spaces to the case of operators on Banach spaces. Next, we consider semi-$C^{*}$-$B$-Fredholm operators on Hilbert $C^{*}$-modules and give sufficient conditions under which the composition of two mutually commuting semi-$C^{*}$-$B$-Fredholm operators is a semi-$C^{*}$-$B$-Fredholm operator, thus generalizing the result by Berkani regarding semi-$B$-Fredholm operators on Banach spaces. Finally, we consider Drazin invertible $C^{*}$-operators, and we give necessary and sufficient conditions for two mutually commuting $C^{*}$-operators to be Drazin invertible when their composition is Drazin invertible.}

\vspace{1.5ex}


\vspace{2.5ex}
\textbf{Acknowledgement}

I am grateful to Professor Dragan S. Djordjevi\'{c} for suggesting me semi-$C^{*}$-Fredholm theory as the topic of my research and for introducing to me the relevant literature. Also, I am grateful to Professor Snezana \v{Z}ivkovi\'{c} Zlatanovi\'{c} for suggesting me to consider Drazin invertible and Browder $C^{*}$-operators. Finally, I am grateful to Professor Vladimir M. Manuilov for the helpful comments regarding the introductory text of the paper. 

This preprint has not undergone peer review (when applicable) or any post-submission improvements or corrections. The Version of Record of this article is published in \cite{IS11}, and is available online at https://doi.org/10.1007/s43034-023-00258-0 .

\newpage

\section{Introduction}

In \cite{DDj2} Djordjevi\'{c} defined generalized Weyl operators on Banach spaces to be the closed range operators for which the kernel and the co-image are isomorphic Banach spaces. He proved then that if the product of two generalized Weyl operators on a Hilbert space has closed image, then this product is also a generalized Weyl operator and he proved that the set of all generalized Weyl operators on a Hilbert space is invariant under finite rank perturbations. However, he remains as an open question whether these statements still hold if we consider operators on general Banach spaces. In this paper we extend these results by Djordjevi\'{c} to the case of operators on Banach spaces.\\

Now, Hilbert $C^{*}$-modules are natural generalization of Hilbert spaces when the field of scalars is replaced by an arbitrary $C^{*}$-algebra. The general theory of Hilbert $C^{*}$-modules was established by Paschke in for instance \cite{P}.

Fredholm theory on Hilbert $C^*$-modules as a generalization of Fredholm theory on Hilbert spaces was started by Mishchenko and Fomenko in \cite{MF}. They have introduced the notion of a Fredholm operator on the standard module over a unital $C^{*}$-algebra and established its main properties.In \cite{IS1} we went further in this direction and defined adjointable semi-$C^{*}$-Fredholm and adjointable semi-$C^{*}$-Weyl operators on Hilbert $C^{*}$-modules. We investigated then and proved several properties of these generalized semi Fredholm and semi-Weyl operators on Hilbert $C^{*}$-modules as an analogue or a generalization of the well-known properties of the classical semi-Fredholm and semi-Weyl operators on Hilbert and Banach spaces. The interest for studying such operators comes from pseudo-differential operators acting on manifolds. The classical theory works nice for compact manifolds, but not for general ones. Even operators on Euclidean spaces are hard to study, e.g. Laplacian is not Fredholm.  However, they sometimes are Fredholm when considered as operators on a compact manifold with coefficients in some $C^{*}$-algebra. Kernels and cokernels of many operators are infinite-dimensional as Banach spaces, but become finitely generated viewed as Hilbert modules. This is the most important reason for studying semi-$C^{*}$-Fredholm operators.

As a part of this research project on semi-$C^{*}$-Fredholm theory, in \cite{IS5} we  define and consider generalized $C^{*}$-Weyl operators, as a generalization (in the setting  of operators on Hilbert $C^{*}$-modules) of generalized Weyl operators defined by Djordjevi\'{c}. We prove in \cite{IS5} for instance an analogue in the setting of generalized $C^{*}$-Weyl operators of the first of the two main theorems by Djordjevi\'{c} in \cite{DDj2} which states that a product of two generalized Weyl operators is also a generalized  Weyl operator in the case when this product has closed range.

In this paper we give a generalization in the setting of generalized $C^{*}$-Weyl operators of the second of  the two main theorems by Djordjevi\'{c} in \cite{DDj2}, the theorem which states that the set of all generalized Weyl operators on a Hilbert space is invariant under finite rank perturbations. \\

Semi-$B$-Fredholm operators have been defined and investigated by Berkani in for instance \cite{BS}, \cite{BM}. The notion of semi-$C^{*}$-$B$-Fredholm operators on Hilbert $C^{*}$-modules as a generalization of semi-$B$-Fredholm operators was introduced in \cite{IS5}. Now, in this paper we present an extension in the setting of semi-$C^{*}$-$B$-Fredholm operators of the well known result by Berkani given in \cite{BM} which states that if two $B$-Fredholm operators mutually commute, then their composition is also $B$-Fredholm and the index is additive.

In several results in this paper we assume that the image of an operator is closed, which shows that closed range operators are important in semi-Fredholm theory on Hilbert $ C^{*}$-modules. This naturally leads  to study closed range $C^{*}$-operators. For two arbitrary bounded, $C^{*}$-operators with closed images, we give  necessary and sufficient conditions under which their composition has closed image. This is a generalization (in the setting  of operators on Hilbert $C^{*}$-modules) of the well-known Bouldin's result in \cite{Bld} for operators on Hilbert spaces. Moreover, we give examples of $C^{*}$-Fredholm operators with non-closed image. Also, we give an example of a $C^{*}$-Fredholm operator $F$ such that $Im F$ is closed, but $Im F^{2} $ is not closed. This illustrates how differently $C^{*}$-Fredholm operators may behave from the classical Fredholm operators on Hilbert and Banach spaces, that always have closed image.\\

Recall that a bounded, linear operator on a Banach space is called Drazin invertible if it has finite ascent and descent whereas it is called Browder if it is both Fredholm and Drazin invertible. In the last section of the paper, we consider Drazin invertible $C^{*}$-operators as a generalization of Drazin invertible operators on Hilbert spaces. We give necessary and sufficient conditions for two mutually commuting $C^{*}$-operators to be Drazin invertible when their composition is Drazin invertible. Also, we give examples of two mutually commuting $C^{*}$-operators whose composition is Drazin invertible whereas they are not Drazin invertible. Finally, we introduce a concept of $C^{*}$-Browder operator as a generalization of  Browder operator on a Hilbert space and study the relationship between Drazin invertible $C^{*}$- operators and $C^{*}$-Browder operators.

Section  \ref{S02} and Section \ref{S03} contain the unpublished results from the PhD thesis by the author, see \cite{IS9}.

\section{Preliminaries}\label{S01}

In this paper we let $\mathcal{A} $ denote a unital $C^{*}$-algebra and $H_{\mathcal{A} }$ denote the standard Hilbert module over $ \mathcal{A}.$ For a Hilbert $\mathcal{A}$-module  $M$ we let $B^{a}(M) $ denote the $C^{*}$-algebra of all $\mathcal{A}$-linear, bounded, adjointable operators on $M.$\\
Moreover, for general Banach spaces $X$ and $Y$, we let $B(X, Y)$ denote the Banach algebra of all bounded, linear operators from $X$ into $Y$ and we simply put $B(X,X):=B(X).$ 	

By the symbol $\tilde{ \oplus} $ \index{$\tilde{ \oplus} $} we denote the direct sum of modules as given in \cite{MT}.

Thus, if $M$ is a Hilbert $C^{*}$-module and $M_{1}, M_{2}$ are two closed submodules of $M,$ we write $M=M_{1} \tilde \oplus M_{2}$ if $M_{1} \cap M_{2}=\lbrace 0 \rbrace$ and $M_{1}+ M_{2}=M.$ If, in addition $M_{1}$  and $M_{2}$ are mutually orthogonal, then we write $M=M_{1} \oplus M_{2}.$

\begin{definition} \label{D D05}    
	\cite[Definition 2.1]{IS1}, \cite{MF} Let $F \in B^{a}(H_{\mathcal{A}}).$ We say that $F $ is an upper semi-{$\mathcal{A}$}-Fredholm operator if there exists a decomposition $$H_{\mathcal{A}} = M_{1} \tilde \oplus {N_{1}} \stackrel{F}{\longrightarrow} M_{2} \tilde \oplus N_{2}= H_{\mathcal{A}} $$ with respect to which $F$ has the matrix\\
	
	\begin{center}
		$\left\lbrack
		\begin{array}{cc}
		F_{1} & 0 \\
		0 & F_{4} \\
		\end{array}
		\right \rbrack,
		$
	\end{center}
	where $F_{1}$ is an isomorphism, $M_{1},M_{2},N_{1},N_{2}$ are closed submodules of $H_{\mathcal{A}} $ and $N_{1}$ is finitely generated. Similarly, we say that $F$ is a lower semi-{$\mathcal{A}$}-Fredholm operator if all the above conditions hold except that in this case we assume that $N_{2}$ ( and not $N_{1}$ ) is finitely generated. If both $N_{1} $ and $N_{2} $ are finitely generated, then $F$ is $\mathcal{A}$-Fredholm operator.
\end{definition}
Set
\begin{center}
	$\mathcal{M}\Phi_{+}(H_{\mathcal{A}})=\lbrace F \in B^{a}(H_{\mathcal{A}}) \mid F $ is upper semi-{$\mathcal{A}$}-Fredholm $\rbrace ,$	\index{$\mathcal{M}\Phi_{+}(H_{\mathcal{A}})$}
\end{center}
\begin{center}
	$\mathcal{M}\Phi_{-}(H_{\mathcal{A}})=\lbrace F \in B^{a}(H_{\mathcal{A}}) \mid F $ is lower semi-{$\mathcal{A}$}-Fredholm $\rbrace ,$	\index{$\mathcal{M}\Phi_{-}(H_{\mathcal{A}})$}
\end{center}
\begin{center}
	$\mathcal{M}\Phi(H_{\mathcal{A}})=\lbrace F \in B^{a}(H_{\mathcal{A}}) \mid F $ is $\mathcal{A}$-Fredholm operator on $H_{\mathcal{A}}\rbrace .$ \index{$\mathcal{M}\Phi(H_{\mathcal{A}})$}
\end{center} 

Next we set $\mathcal{M}\Phi_{\pm}(H_{\mathcal{A}})=\mathcal{M}\Phi_{+}(H_{\mathcal{A}}) \cup \mathcal{M}\Phi_{-}(H_{\mathcal{A}}) .$ \index{$\mathcal{M}\Phi_{\pm}(H_{\mathcal{A}})$}
Notice that if $M,N$ are two arbitrary Hilbert modules $C^{*}$-modules, the definition above could be generalized to the classes $\mathcal{M}\Phi_{+}(M,N)$ and $\mathcal{M}\Phi_{-}(M,N)$.

\begin{definition} \label{DP D06}  
	\cite{KM} \cite[Definition 2.7.1]{MT} Let $M $ be an abelian monoid. Consider the Cartesian product $M \times M $ and its quotient monoid with respect to the equivalence relation
	$$(m,n) \sim (m^{\prime}, n^{\prime}) \Leftrightarrow \exists p,q:(m,n)+(p,p)=(m^{\prime}, n^{\prime})+(q,q).$$
	This quotient monoid is a group, which is denoted by $S(M)$ and is called the symmetrization of $M.$ Consider now the additive category $\mathcal{P}(\mathcal{A}) $ of projective modules over a unital $C^{*}$-algebra $\mathcal{A}$ and denoted by $[\mathcal{M}] $ the isomorphism class of an object $\mathcal{M} $ from $\mathcal{P}(\mathcal{A}) .$ The set $\phi(\mathcal{P}(\mathcal{A})) $ of these classes has the structure of an Abelian monoid with respect to the operation $[\mathcal{M}]+[\mathcal{N}]=[\mathcal{M}  \oplus \mathcal{N}] .$ In this case the group $S(\phi(\mathcal{P}(\mathcal{A}))) $ is denoted by $K(\mathcal{A}) $ or $K_{0}(\mathcal{A}) $ \index{$K_{0}(\mathcal{A}) $} and is called the $K$-group of  $\mathcal{A}$ or the Grothendieck group of the category $\mathcal{P}(\mathcal{A}) .$
\end{definition} 

As regards the $K$-group $K_{0} (\mathcal{A})$, it is worth mentioning that it is not true in general that $[M]=[N]$ implies that $ M \cong N    $ for two finitely generated Hilbert modules $M, N$ over $\mathcal{A}$. If $K_{0} (\mathcal{A})$ satisfies the property that $[N]=[M]$ implies that $ N \cong M     $ for any two finitely generated, Hilbert modules $M, N$ over $\mathcal{A}$, then $K_{0} (\mathcal{A})   $ is said to satisfy "the cancellation property", see \cite[Section 6.2] {W}.

\begin{definition} \label{DP D08} 
	\cite[Definition 2.7.8]{MT}
	Let $F \in \mathcal{M}{\Phi} (H_{\mathcal{A}}).$ We define the index of $F$ by 
	$$ \text {\rm index }  F=[\mathcal{N}_{1}]-[\mathcal{N}_{2}] \in K_{0}(\mathcal{A}).$$ \index{$\text {\rm index }  F$}
\end{definition}

\begin{theorem}  \label{DP T07}    
	\cite[Theorem 2.7.9]{MT} The index is well defined.
\end{theorem}

\begin{lemma}  \label{D L18}        
	Let $M$ be a Hilbert $C^{*}$-module and $M_{1}, M_{2} $ be closed submodules of $M$ such that $M_{1} \subseteq M_{2} $ and  $M=M_{1} \tilde \oplus M_{1}^{\prime} $ for some Hilbert submodule $ M_{1}^{\prime}.$ Then $M_{2} =M_{1} \tilde \oplus (M_{1}^{\prime} \cap M_{2}).$
\end{lemma}	

\begin{proof}
	Since $M=M_{1} \tilde \oplus M_{1}^{\prime}	$ by assumption and $M_{2} \subseteq M,$ any $z \in M_{2}$ can be written as $z=x+y$ for some $x \in M_{1}$ and $y \in M_{1}^{\prime}.$ Now, since $M_{1} \subseteq M_{2} $ by assumption, we have $y=z-x \in M_{2}.$  Thus, $y \in M_{1}^{\prime} \cap M_{2}.$
\end{proof}

\begin{remark}
	Lemma \ref{D L18} is a slightly modifed version of \cite[Lemma 2.6]{IS3}.
\end{remark}

\section{Generalized $C^{*}$-Weyl operators}\label{S02}
In this section we consider generalized $C^{*}$-Weyl operators and provide a generalization in this setting of  \cite[Theorem 2]{DDj2} concerning perturbations of generalized Weyl operators by finite rank operators. Moreover, we extend the results in \cite{DDj2} from the case of operators on Hilbert spaces to the case of regular operators on Banach spaces.\\
We start with the following definition.

\begin{definition} \label{06d 13}  
	For two Hilbert $C^{*}$-modules $M$ and $M^{\prime}$ we set ${\tilde{\mathcal{M}}\Phi}_{0}^{gc}  (M,M^{\prime})   $ \index{$\tilde{\mathcal{M}\Phi}_{0}^{gc}  (M,M^{\prime})$} to be the class of all closed range operators $F \in B^{a} (M,M^{\prime}) $ for which there exist finitely generated Hilbert submodules $N, \tilde{N}  $ with the property that $$N \oplus \ker F \cong \tilde{N} \oplus Im F^{\perp}  .$$
\end{definition}

Then we obtain the following generalization of \cite[Theorem 2]{DDj2}.

\begin{lemma} \label{06l 09}   
	Let $T \in  {\tilde{\mathcal{M}}\Phi}_{0}^{gc}  (H_{\mathcal{A}} )   $ and $ F \in  B^{a}(H_{\mathcal{A}} )$  such that  $ Im F$ is closed, finitely generated. Suppose that $Im (T+F), T (\ker F), P (\ker T), P (\ker (T+F))$ are closed, where $P$ denotes the orthogonal projection onto $\ker F^{\perp}  .$ Then 
	$$T+F \in {\tilde{\mathcal{M}}\Phi}_{0}^{gc}  (H_{\mathcal{A}} )  .$$
\end{lemma}

\begin{proof} 
	Since $Im T$ and $Im (T+F)$ are closed by assumption, by \cite[Theorem 2.3.3]{MT} we have $H_{\mathcal{A}} = Im T \oplus Im T^{\perp}  $ and $H_{\mathcal{A}} = Im (T+F) \oplus Im (T+F)^{\perp}.$ Similarly, since $ Im F $ is closed by assumption, from \cite[Theorem 2.3.3]{MT} we get that \linebreak[4] 
	$ H_{\mathcal{A}}=\ker F^{\bot}\oplus\ker F. $ Hence $ T_{\mid_{\ker F}} $ is an adjointable operator from $ \ker F $ into $ Im T $ (and $ (T+F)_{\mid_{\ker F}}=T_{\mid_{\ker F}} $ is an adjointable operator from $ \ker F $ into $ Im (T+F) $ ). Now, since $ T(\ker F) $ is closed by assumption, again by applying  \cite[Theorem 2.3.3]{MT} on the operator $ T_{\mid_{\ker F}}  ,$ we deduce that 
	$$Im T=T (\ker F)\oplus N \text{ and } \text{ }Im (T+F)=T (\ker F) \oplus N^{\prime}$$ for some Hilbert submodules $ N,N^{\prime} .$ Hence $$ Im T^{\perp} \oplus N=Im(T+F)^{\perp} \oplus N^{\prime}=T(\ker F)^{\perp} .$$ 
	Thus, $T(\ker F) $ is orthogonally complementable in $H_{ \mathcal{A}}.$ Let $ Q$ denote the orthogonal projection onto $T(\ker F)^{\perp}.$
	It turns out that $N$ and $N^{\prime} $ are finitely generated. Indeed, we have 
	$$Im T = T (\ker F)+T(\ker F^{\perp}) \text{ and } Im (T + F)=T(\ker F)+(T+F)(\ker F^{\perp}).$$ 
	As $F_{\mid_{\ker F^{\perp}}} $ is an isomorphism onto $Im F$ by the Banach open mapping theorem and $Im F$ is finitely generated by assumption, it follows that $\ker F^{\perp} $ is finitely generated. Hence $QT(\ker F^{\perp}) $ and $Q(T+F)(\ker F^{\perp}) $ are finitely generated. However, we have 
	$$N=Q(Im T)=QT(\ker F^{\perp}) \text{ and } N^{\prime}=Q(Im (T+F))=Q(T+F)(\ker F^{\perp}).$$ Furthermore, since $P (\ker T)$ is closed by assumption and $P_{\mid_{\ker T}}  $ is adjointable (as $\ker T$ is orthogonally complementable by \cite[Theorem 2.3.3]{MT}), then $\ker P_{\mid_{{(\ker T)}}}$ $=\ker F \cap \ker T  $ is orthogonally complementable in $\ker T,$ so 
	$$\ker T=(\ker F \cap \ker T) \oplus M  $$ 
	for some closed submodule $M.$ We have that $P_{\mid_{M}}  $ is an isomorphism onto $P(\ker T).$  Since $P_{\mid_{\ker T}}  $ is adjointable and $P (\ker T)$ is closed, by \cite[Theorem 2.3.3]{MT} $ P (\ker T)$ is orthogonally complementable in $\ker F^{\perp}  .$ As $\ker F^{\perp} $ is finitely generated, it follows that $P (\ker T)$ is finitely generated. Thus, $M$ must be finitely generated because $P_{\mid_{M}}  $ is an isomorphism onto $P(\ker T).$ \\
	By similar arguments as above, using that $ P(\ker (T+F)) $ is closed by assumption, we obtain that 
	$$\ker (T+F)= (\ker(T+F) \cap \ker F ) \oplus M^{\prime} ,$$ 
	where $M^{\prime}$ is a finitely generated Hilbert submodule. Now, 
	$ \ker T \cap \ker F = \ker (T+F) \cap \ker F,$ 
	so we have 
	$$\ker (T+F)=(\ker T \cap \ker F) \oplus M^{\prime}.$$ 
	Finally, since $T \in  {\tilde{\mathcal{M}}\Phi}_{0}^{gc}  (H_{\mathcal{A}} )   ,$ there exist finitely generated Hilbert submodules $R$ and $R^{\prime}$ such that $R \oplus \ker T \cong R^{\prime} \oplus Im T^{\perp}.$ Combining all this together, we deduce that 
	$$\ker (T+F) \oplus M \oplus N \oplus R \cong (\ker T \cap \ker F) \oplus M^{\prime} \oplus M \oplus N \oplus R  $$ $$\cong \ker T \oplus M^{\prime} \oplus N \oplus R \cong Im T^{\perp} \oplus M^{\prime} \oplus N \oplus R^{\prime} \cong Im(T+F)^{\perp} \oplus M^{\prime} \oplus N^{\prime} \oplus R^{\prime} .$$
\end{proof}

Next we recall the definition of generalized $\mathcal{A}$-Weyl operator. 

\begin{definition} \label{D01}   
	\cite[Definition 11]{IS5} Let $F \in B^{a} (H_{\mathcal{A}}) .$	We say that $F$ is generalized $\mathcal{A}$-Weyl, denoted by $F \in  {\mathcal{M}\Phi}_{0}^{gc}(H_{\mathcal{A}})   $ \index{$\mathcal{M}\Phi_{0}^{gc}(H_{\mathcal{A}})$} if $Im F$ is closed and $ker F \cong Im F^{\perp}.$ 
\end{definition}

From Lemma \ref{06l 09} we deduce the following corollary.

\begin{corollary} \label{06c 05}   
	Let $T \in {\mathcal{M}\Phi}_{0}^{gc}  (  H_{\mathcal{A}}   )    $ and suppose that $\ker T \cong Im T^{\perp} \cong H_{\mathcal{A}} .$ If $F \in B^{a} (H_{\mathcal{A}} )  $ satisfies the assumptions of Lemma \ref{06l 09}, then 
	$$ \ker(T+F) \cong Im(T+F)^{\perp} \cong H_{\mathcal{A}} .$$ 
	In particular, $T+F \in {\mathcal{M}\Phi}_{0}^{gc}  (H_{\mathcal{A}} )  .$ 
\end{corollary}

\begin{proof}
	Notice that, since $T \in {\mathcal{M}\Phi}_{0}^{gc}  (  H_{\mathcal{A}}   )  $ by hypothesis, we already have that $\ker T \cong Im T^{\perp}  ,$ so the additonal assumption is that $\ker T$ and $Im T^{\perp}$  are isomorphic to $H_{\mathcal{A}}.$ 
	By the proof of Lemma \ref{06l 09} (and using the same notation), since $\mathcal{M} \Phi_{0}^{gc}  (  H_{\mathcal{A}} ) \subseteq \tilde{\mathcal{M}} \Phi_{0}^{gc}  (  H_{\mathcal{A}} ),$ we have 
	$$\ker(T+F) \oplus M \oplus N \oplus R \cong \ker T \oplus M^{\prime} \oplus N \oplus R $$ $$\cong Im T^{\perp} \oplus M^{\prime} \oplus N \oplus R^{\prime} \cong Im (T+F)^{\perp} \oplus M^{\prime} \oplus N^{\prime} \oplus R^{\prime}  .$$ 
	Since $M,N,R,M^{\prime},N^{\prime},R^{\prime}  $ are finitely generated Hilbert submodules and $\ker T \cong Im T^{\perp} \cong H_{\mathcal{A}}  $ by assumption, by the Kasparov stabilization Theorem \cite[Theorem 1.4.2]{MT} we have $$H_{\mathcal{A}} \cong \ker T \oplus M^{\prime} \oplus N \oplus R \cong Im T^{\perp} \oplus M^{\prime} \oplus N \oplus R^{\prime}  .$$ 
	Hence $$H_{\mathcal{A}} \cong \ker(T+F) \oplus M \oplus N \oplus R \cong Im (T+F)^{\perp} \oplus M^{\prime} \oplus N^{\prime} \oplus R^{\prime}    .$$ 
	By the Dupre-Filmore  Theorem \cite[Theorem 1.4.5]{MT}, it follows easily that 
	$$\ker(T+F) \cong Im (T+F)^{\perp} \cong H_{\mathcal{A}}  .$$ 
\end{proof}

\begin{corollary}   \label{11r 025}        
	Let $H$ be a separable infinite dimensional Hilbert space and $T$ be a generalized Weyl operator on $H.$ If $\ker T (\cong Im T^{\perp}) $ is infinite  dimensional, and $F$ is a finite rank operator on $H,$ then $T+F$ is generalized Weyl.
\end{corollary}

\begin{proof}
	Since $Im T= T(\ker F)+ T (\ker F^{\perp}) $ and $dim T (\ker F^{\perp}) < \infty, $ by Kato Theorem  \cite[Corollary 1.1.7]{ZZRD} applied on the operator $T_{\mid_{\ker F}}:\ker F \rightarrow Im T ,$ we get that $T (\ker F)$ is closed. Hence, since $Im (T+F)=T(\ker F)+(T+F)(\ker F^{\perp})  $ and $\dim (T+F) (\ker F^{\perp})< \infty,$  by  \cite[Lemma 1.1.2]{ZZRD} we must have that $Im (T+F)$ is closed. 
	
	Finally, since $P$ is finite rank operator (where $P$ is the orthogonal projection onto $\ker F^{\perp},$ it follows that $P (\ker T)$ and $P (\ker (T+F))$ are closed. Therefore, by Corollary \ref{06c 05} we conclude that $T+F$ is generalized Weyl.
\end{proof}

\begin{remark}
	Corollary \ref{11r 025} is actually the main statement in \cite[Theorem 2]{DDj2} Indeed, if $\ker T(\cong Im T^{\perp}) $ is finite dimensional, then $T$ is Weyl in the classical sense, so it is well known that $T+F$ is also Weyl in this case. Therefore, the proof of \cite[Theorem 2]{DDj2} deals only with the situation when $\ker T(\cong Im T^{\perp}) $ is infinite dimensional.
\end{remark}

\begin{lemma} \label{06l 10}    
	Let $T \in \mathcal{M}\Phi (H_{\mathcal{A}} )$  and suppose that $Im T$ is closed. Then
	$ T \in {\tilde{\mathcal{M}}\Phi}_{0}^{gc}  (H_{\mathcal{A}} )  .$
\end{lemma}

\begin{proof} 
	By \cite[Lemma 12]{IS5}, since $Im T$ is closed and $T \in \mathcal{M}\Phi (H_{\mathcal{A}} ) ,$ we have that $ \ker T \text{ and } Im T^{\perp}$ are then finitely generated. By \cite[Theorem 2.7.5]{MT} we can find an $n \in \mathbb{N}  $  such that  
	$$L_{n}=P \tilde{\oplus} p_{n}(\ker T)=P^{\prime} \tilde{\oplus} p_{n}(Im T^{\perp})$$  and 
	$$ p_{n} (\ker T) \cong \ker T,p_{n} (Im T^{\perp}) \cong Im T^{\perp},$$ 
	where $P$ and $P^{\prime}$ are finitely generated Hilbert submodules and $p_{n}$ denotes the orthogonal projection onto $L_{n}$. It follows that $P \oplus \ker T \cong P^{\prime} \oplus Im T^{\perp} .$
\end{proof}

We present now the definition of regular operators on Banach spaces.

\begin{definition} \label{06d 21}    
	Let $X, Y$ be Banach spaces and $T \in B(X,Y). $ Then $T$ is called a regular operator if $T(X)$ is closed in $Y$ and in addition $ T^{-1}(0)$ and $T(X)$ are complementable in $X$ and $Y,$ respectively.
\end{definition}

\begin{remark} \label{D r05}       
	It is not hard to see that $T$ is a regular operator if and only if $T$ admits a generalized inverse, that is if and only if there exists some $T^{\prime} \in B(Y,X) $ such that $TT^{\prime}T=T.$ In this case we have that $TT^{\prime}$ and $T^{\prime}T$ are the projections onto $T(X)$ and complement of $T^{-1}(0),$ respectively, and moreover, $T^{\prime}TT^{\prime}=T^{\prime}.$ Thus, Definition \ref{06d 21} corresponds to the definition of regular operators on Banach spaces given in \cite{H}.
\end{remark}

We can apply the arguments from the proof of Lemma \ref{06l 09} to obtain an extension of \cite[Theorem 2]{DDj2} to the case of regular operators on Banach spaces.\\
First we give the following definition.

\begin{definition}
	Let $X, Y$ be Banach spaces. We set $\Phi_{0}^{gc}(X,Y) $ \index{$\Phi_{0}^{gc}(X,Y) $}  to be the set of all regular operators $T \in B(X,Y) $ satisfying that there exist finite dimensional Banach spaces $Z_{1} $ and $ Z_{2}$ with the property that $\ker T \oplus Z_{1} \cong Im T^{\circ} \oplus Z_{2},$ where $Im T^{\circ} $ stands for the complement of $Im T$ in $Y.$ 
\end{definition}

Then we give the following extension of \cite[Theorem 2]{DDj2} to the case of regular operators on Banach spaces. 

\begin{lemma} \label{DP L40}    
	Let $X, Y$ be Banach spaces and $T \in \Phi_{0}^{gc}(X,Y).$ Suppose that $F$ is a finite rank operator from $X$ into $Y.$ Then $T+F \in \Phi_{0}^{gc}(X,Y).$
\end{lemma}

\begin{proof}
	Since $F$ is finite rank operator, it is regular, i.e. $Im F$ is closed, $\ker F$ and $Im F$ are complementable in $X$ and $Y,$ respectively. Let $\ker F^{\circ}$ denote complement of $\ker F$ in $X.$ As $Im T$ is closed by assumption and $Im T=T(\ker F)+T(\ker F^{\circ}),$ it follows that $T(\ker F)$ has finite co-dimension in $Im T,$ so, by the Kato Theorem \cite[Corollary 1.1.7]{ZZRD},  we have that $T (\ker F)$ is closed ( as $T(\ker F)=Im T_{\mid_{\ker F}}$ and $\ker F^{\circ} $ is finite dimensional). Hence, again using that $T(\ker F)$ has finite co-dimension, by part b) in \cite[Lemma 4.21]{RW} we obtain that $Im T=T(\ker F) \tilde \oplus N  ,$ where $N$ is a finite dimensional subspace. Now, since $T(\ker F)$ is closed and $$Im(T+F)=T(\ker F)+(T+F)(\ker F^{\circ}),$$
	by \cite[Lemma 1.1.2]{ZZRD}  we get that $Im (T+F)$ is closed as $ (T+F)(\ker F^{\circ})$ is finite dimensional.  By the similar arguments as above, we deduce then that $ Im(T+F)=T(\ker F) \tilde \oplus N^{\prime}$ for some finite dimensional subspace $N^{\prime}.$ Since 
	$$Y=Im T  \tilde \oplus Im T^{\circ}=T(\ker F) \tilde \oplus N \tilde \oplus Im T^{\circ} ,$$ where $Im T^{\circ} $ 
	stands for the complement of $Im T$ in $Y$, we see that $T (\ker F)$ is complementable in $Y.$ Let $T(\ker F)^{\circ} $ denote complement of $T (\ker F)$ in $Y$ and $Q$ be the projection onto $T(\ker F)^{\circ} $ along $T (\ker F).$ Then $Q_{\mid_{N^{\prime}}} $ is injective. As $N^{\prime} $ is finite dimensional, so is $Q(N^{\prime}),$ hence $Q(N^{\prime}) $ is closed and $T(\ker F)^{\circ}=Q(N^{\prime}) \tilde \oplus V $ for some closed subspace $V.$ 
	This follows by part a) in \linebreak[4] \cite[Lemma 4.21]{RW} . Since $Q_{\mid_{N^{\prime}}} $ is then an isomorphism onto $Q(N^{\prime}),$ by the same arguments as in the proof of \cite[Proposition 3]{IS5} we deduce that $$Y=T(\ker F) \tilde \oplus N^{\prime}  \tilde \oplus V = Im (T+F) \tilde \oplus V,$$
	so $Im (T+F)$ is complementable. 
	
	Next, let $P$ denote the projection onto $\ker F^{\circ} $ along $\ker F.$ Then $P_{\mid_{\ker T}} $ and $P_{\mid_{\ker (T+F)}} $ are finite rank operators, hence regular. It follows that their kernels are complementable, hence by the same arguments	as in the proof of Lemma \ref{06l 09} we deduce that 
	$$\ker T=(\ker T \cap \ker F) \tilde \oplus M \text{ and } \ker (T+F)=(\ker T \cap \ker F) \tilde \oplus M^{\prime} $$ for some finite dimensional subspaces $M$ and $M^{\prime}.$ Since $\ker T$ is complementable in $X$ as $T$ is regular, then $\ker T \cap \ker F$  is complementable in $X,$ so by the similar arguments as above we can deduce that $\ker (T+F)$ is complementable in $X.$ Hence $T+F$ is a regular operator. Moreover, proceeding in the same way as in the proof of Lemma \ref{06l 09} by considering chain of isomorphisms, we conclude that $T+F \in \Phi_{0}^{gc}(X,Y).$
\end{proof}

\begin{remark}  \label{DP R16}   
	If $H$ is a Hilbert space, it follow that if $F \in \Phi_{0}^{gc}(H)    $ and $\ker F $ or $Im F^{\perp}$ are infinite-dimensional, then $\ker F \cong Im F^{\perp}  .$ Hence it is not hard to see that Lemma \ref{DP L40} is indeed an extension of \cite[Theorem 2]{DDj2}.
\end{remark}

Next we recall the definition of generalized Weyl operators on Banach spaces. 

\begin{definition}  \label{06d 22} 
	\cite{DDj2} Let $X, Y$ be Banach spaces and $T \in B(X,Y).$ Then we say that $T$ is generalized Weyl, if $T(X)$ is closed in $Y,$ and $T^{-1}(0)$ and $Y{/{T(X)}} $ are mutually isomorphic Banach spaces. 
\end{definition}

We give then the following proposition as an extension of \cite[Theorem 1]{DDj2} to the case of regular operators on Banach spaces.

\begin{proposition} \label{DP P16}   
	Let $X, Y, Z$ be Banach spaces and let  $ T \in  B(X, Y), S \in B(Y, Z).$ Suppose that $T, S, ST$ are regular, that is $T (X), S (Y), ST (X)$ are closed and $T, S, ST$ admit generalized inverse. If  $T$ and $S$ are  generalized Weyl operators, then $ST$ is a generalized Weyl operator.
\end{proposition} 

\begin{proof} 
	Since $T, S, ST$ are regular by assumption, their kernels and ranges are complementable in the respective Banach spaces $X, Y, Z.$ Moreover, observe that  $S_{\mid_{T(X)}}  $ is regular. Indeed, if $ U $ denotes the generalized inverse of $ST,$  then for any $x$ in $X,$ we have $S T U S T (x)= ST (x),$ so it is easily seen that $T U$ is generalized inverse of $S_{\mid_{T(X)}} . $ Hence $(S_{\mid_{T(X)}})^{-1} (0) $ is complementable in $T(X).$  However, we have $(S_{\mid_{T(X)}})^{-1} (0) = S^{-1}(0) \cap T(X).$ Since $T(X)$ is complementable in $Y,$ because $T$ is regular, it follows that $ S^{-1}(0) \cap T(X)$ is complementable in $Y.$ By Lemma \ref{D L18} we have that $ S^{-1}(0) \cap T(X)  $ is then complementable in $ S^{-1}(0) .$  Moreover, $ST(X)$ is complementable in $S(Y)$ by Lemma \ref{D L18}, since $ST(X)$ is complementable in $Z.$  Finally, since $T^{-1}(0)  $ is complementable in $X,$ because $T$ is regular, and $T^{-1}(0) \subseteq ST^{-1}(0)  ,$ it follows again from Lemma \ref{D L18} that $T^{-1}(0)  $ is complementable in $ST^{-1}(0)  .$ Then we are in the position to apply exactly the same proof as in \cite[Proposition 3]{IS5}. 
\end{proof}

\begin{remark} \label{06r 02}   
	In general, if $X, Y, Z$ are Banach spaces and $F \in B(X,Y), G \in B(Y,Z),$ $GF \in B(X,Z)$ are regular operators, then we have that the sequence 
	$$0 \rightarrow \ker F \rightarrow \ker GF  \rightarrow \ker G  \rightarrow Im F^{\circ} \rightarrow Im GF^{\circ} \rightarrow Im G^{\circ} \rightarrow 0 $$ 
	is exact, where $ImF^{\circ},ImG^{\circ}$ and $ImGF^{\circ}$ denote the complements of $Im F,ImG $ and $ ImGF$ in the  respective Banach spaces. This can be deduced from the proof of \cite[Proposition 3]{IS5} and Proposition \ref{DP P16} or from \cite[Proposition 2.1]{KY} and \cite[Theorem 2.7]{KY}. If $G, F, GF$ are regular operators, then all the subspaces in the above sequence are complementable in the respective Banach spaces. From the exactness of the above sequence we may deduce as direct corollaries various results such as \cite[Theorem 1]{DDj2} and index theorem, Harte's ghost theorem in \cite{H} etc. 
\end{remark}

\begin{lemma}  \label{DP L42}    
	Let $ \tilde M   $ be a Hilbert $C^{*}$-module and $ F,D \in \tilde{\mathcal{M}\Phi}_{0}^{gc}(\tilde M) .$ If $Im DF$ is closed, then $DF \in \tilde{\mathcal{M}\Phi}_{0}^{gc}(\tilde M)  .$
\end{lemma}

\begin{proof}
	Since $ F,D \in \tilde{\mathcal{M}\Phi}_{0}^{gc}(\tilde M)    $ by assumption, there exist finitely generated Hilbert submodules $  N, \tilde N ,N^{\prime}  $ and $ \tilde N^{\prime}   $ such that 
	$$N \oplus \ker F \cong \tilde N \oplus Im F^{\perp}  \text{ and }  N^{\prime} \oplus \ker D \cong \tilde N^{\prime} \oplus Im D^{\perp}  .$$ 
	By applying the arguments from the proof of \cite[Proposition 3]{IS5} and using the same notation, we obtain the following chain of isomorphisms:
	$$\ker DF \oplus N \oplus N^{\prime} \cong \ker F \oplus (\ker D \cap \ker F) \oplus N \oplus N^{\prime} $$ 
	$$\cong Im F^{\perp} \oplus (\ker D \cap Im F) \oplus \tilde N \oplus N^{\prime}   \cong          
	S(X) \oplus M \oplus (\ker D \cap F) \tilde \oplus \tilde N \oplus N^{\prime} $$
	$$\cong  S(X) \oplus \ker D \oplus \tilde N \oplus N^{\prime} \cong X \oplus Im D^{\perp} \oplus \tilde N \oplus \tilde N^{\prime} \cong Im DF^{\perp} \oplus \tilde N \oplus \tilde N^{\prime}.$$ 
\end{proof}

\begin{remark}  \label{DP R17}  
	As explained in the proof of Proposition \ref{DP P16} and Remark \ref{06r 02}, the proof of Proposition \cite[Proposition 3]{IS5} applies in the case of regular operators on Banach spaces. By combining this fact with the proof of Lemma \ref{DP L42} we can deduce that if $ T \in \Phi_{0}^{gc}(X,Y), S \in \Phi_{0}^{gc}(Y,Z)   $ and $ST$ is regular, then $ ST \in \Phi_{0}^{gc}(X,Z)   $ (where $X, Y$ and $Z$ are Banach spaces).
\end{remark}

\section{Semi-$C^{*}$-B-Fredholm operators}\label{S03}
In this section we consider semi-$C^{*}$-$B$-Fredholm operators and provide a generalization in this setting of \cite[Theorem 3.2]{BM} concerning compositions of semi-$B$-Fredholm operators. Moreover, we give necessary and sufficient conditions for composition of two closed range $C^{*}$-operators to have closed image. Also, we introduce examples of $C^{*}$-Fredholm operators with non-closed image.\\
First we recall the following definition.

\begin{definition} \label{D D08}   
	\cite[Definition 16]{IS5} Let $M$ be a Hilbert $\mathcal{A}$-module and $F \in B^{a}(M)  .$ Then $F$ is said to be an upper semi-$\mathcal{A}$-$B$-Fredholm operator if there exists some $n \in \mathbb{N} $ such that $Im F^{m} $ is closed for all $m \geq n$ and $F_{\mid_{Im F^{n}}}  $ is an upper semi-$\mathcal{A}$-Fredholm operator. Similarly, $F$ is said to be a lower semi-$\mathcal{A}$-$B$-Fredholm operator if the above conditions hold except that in this case we assume that  $F_{\mid_{Im F^{n}}}  $ is a lower semi-$\mathcal{A}$-Fredhlom operator and not an upper semi-$\mathcal{A}$-Fredholm operator.
\end{definition}

\begin{lemma}  \label{11r l01}        
	Let $M$ be a Hilbert $C^{*}$-module and $ F \in  {\mathcal{M}\Phi}  (M) .$ If $Im F$ is closed, then the index of $F$ is well defined. 
\end{lemma}

\begin{proof}
	Let
	$$M=M_{1} \tilde{\oplus} N_{1} \stackrel{ {F}}{\longrightarrow} M_{2} \tilde{\oplus} N_{2} =M $$
	be an $\mathcal{M}\Phi$-decomposition for $F.$ Since $N_{1} $ is finitely generated, it is self-dual, hence $F_{\mid_{N_{1}}} $ is adjointable by \cite[Corollary 2.5.3]{MT}. It is not hard to see that $F(N_{1})=ImF \cap N_{2} ,$ hence $F(N_{1}) $ is closed. By \cite[Theorem 2.3.3]{MT}, $F(N_{1})$ is orthogonally complementable in $N_{2} ,$ so $N_{2}= F(N_{1}) \oplus \tilde{N_{2}}$ for some closed submodule $\tilde{N_{2}} .$ Moreover, $\ker F_{\mid_{N_{1}}} $ is orthogonally complementable in $N_{1} $ again by \cite[Theorem 2.3.3]{MT}. Now, it is not hard to see that $\ker F_{\mid_{N_{1}}}=\ker F ,$ so we have that $N_{1}=\ker F \oplus  \tilde{N_{1}}  $ for some closed submodule $ N_{1}.$ Clearly, $F$ maps $\tilde{N_{1}} $ isomorphically onto $F(N_{1}) .$ In addition, by \cite[Theorem 2.3.3]{MT} we have 
	$$M=Im F \oplus Im F^{\perp}=M_{2} \tilde{ \oplus} F(N_{1}) \tilde{ \oplus} \tilde{N_{2}}=Im F \tilde{ \oplus} \tilde{N_{2}} ,$$ 
	which gives $Im F^{\perp} \cong \tilde{N_{2}}.$ Since $\tilde{N_{1}} \cong F(N_{1}) ,$ we get $$[N_{1}]-[N_{2}]=[\ker F] - [ \tilde{N_{2}}]=[\ker F] - [Im F^{\perp}] $$ 
	in $\mathcal{K}_{0} (\mathcal{A}) ,$ so $index F$ is independent of $\mathcal{M}\Phi$-decomposition for $F.$
	
\end{proof}

In \cite{IS5}the index of a $\mathcal{A}$-$B$-Fredholm operator $F$ on $H_{\mathcal{A}} $ is defined as $index F:= index F_{\mid_{ImF^{n}}}  $ where $n$ is such that $ImF^{m} $ is closed for all $m \geq n $ and $F_{\mid_{ImF^{n}}} $ is $\mathcal{A}$-Fredholm operator on $Im F^{n}.$ By \cite[Proposition 7]{IS5} if $Im F^{n}\cong H_{\mathcal{A}},$ then $index F$ is well defined. However, by applying Lemma  \ref{11r l01} in the proof of \cite[Proposition 7]{IS5} we can see that $index F$ is also well defined when $F$ is an $\mathcal{A}$-$B$-Fredholm operator on an arbitrary Hilbert $\mathcal{A}$-module $M$ and when $Im F^{n} $ is not isomorphic to $H_{\mathcal{A}} .$ Below we present a generalization of \cite[Theorem 3.2]{BM} in the setting of $C^{*}$-$B$-Fredholm operators. 

\begin{proposition} \label{DP P26}    
	Let $M$ be a Hilbert-module and $F,D \in B^{a}(M) $ satisfying that $FD=DF.$ Suppose that there exists an $n \in \mathbb{N} $ such that $Im (DF)^{m}$ is closed for all $m \geq n $ and in addition for each $m \geq n  $ we have that $Im F^{m+1}D^{m} $ and $Im D^{m+1}F^{m} $ are closed. If $F$ and $D$ are upper (lower) semi-$\mathcal{A}$-$B$-Fredholm, then $DF$ is upper (lower) semi-$\mathcal{A}$-$B$-Fredholm. If $F$ and $D$ are $\mathcal{A}$-$B$-Fredholm, then $DF$ is $\mathcal{A}$-$B$-Fredholm and $\text {\rm index }  DF=\text {\rm index }  D+\text {\rm index }  F.$ 
\end{proposition}

\begin{proof}
	If $F$ and $D$ are upper semi-$\mathcal{A}$-$B$-Fredholm, then by \cite[Proposition 7]{IS5}  we can choose an $n \in \mathbb{N} $ sufficiently large such that $n$ satisfies the assumption in the proposition and in addition satisfies that $Im D^{m}, Im F^{m} $ are closed and $F_{\mid_{ImF^{m}}}, D_{\mid_{ImF^{m}}} $ are upper semi-$\mathcal{A}$-Fredholm for all $m \geq n  .$ As $Im F^{n+1}D^{n}=Im F(DF)^{n} ,$  $Im D^{n+1} F^{n}=Im D(DF)^{n} ,$ $Im (DF)^{n}$ and $Im (DF)^{n+1}$  are all closed by assumption, we have that $F_{\mid_{Im(DF)^{n}}}, D_{\mid_{Im(DF)^{n}}} $ and $DF_{\mid_{Im(DF)^{n}}} $ are regular operators. This follows from \cite[Theorem 2.3.3]{MT}. Hence we can apply the exact sequence from \cite[Lemma 2]{IS5}. Since $F_{\mid_{ImF^{n}}} $ and $D_{\mid_{ImD^{n}}} $ are upper semi-$\mathcal{A}$-Fredholm, we have that 
	$$\ker F_{\mid_{ImF^{n}}} = \ker F \cap Im F^{n} \text{ and } \ker D_{\mid_{ImD^{n}}} = \ker D \cap Im D^{n} $$  
	are both finitely generated by \cite[Lemma 12]{IS5}. As $F_{\mid_{Im(DF)^{n}}} $ and  $ D_{\mid_{Im(DF)^{n}}} $ are regular operators, it follows that 
	$$\ker F_{\mid_{Im(DF)^{n}}} = \ker F \cap  Im (DF)^{n} \text{ and } \ker D_{\mid_{Im(DF)^{n}}} = \ker D \cap Im(DF)^{n} $$ 
	are both orthogonally complementable in $Im(DF)^{n}.$ However, $Im(DF)^{n} $ is orthogonally complementable in $M$ by \cite[Theorem 2.3.3]{MT}, so $\ker F \cap Im(DF)^{n} $ and $\ker D \cap Im(DF)^{n} $ are orthogonally complementable in $M.$ Since $$Im(DF)^{n}=Im D^{n} F^{n} =Im F^{n} D^{n} \subseteq Im F^{n} \cap Im D^{n} ,$$  
	we get that $$\ker D \cap Im (DF)^{n} \subseteq \ker D \cap Im D^{n} \text{ and } \ker F \cap Im (DF)^{n} \subseteq \ker F \cap Im F^{n} .$$ By Lemma \ref{D L18}  we obtain that	$\ker F \cap Im (DF)^{n} $ and $\ker D \cap Im (DF)^{n}  $ are orthogonally complementable in $\ker F \cap Im F^{n} $ and $ \ker D \cap Im D^{n}  ,$ respectively. As $\ker F \cap Im F^{n} $ and $\ker D \cap Im D^{n}  $ are finitely generated, it follows that $\ker F \cap Im (DF)^{n} \text{ and } \ker D \cap Im (DF)^{n}  $ are both finitely generated. By applying the exact sequence from \cite[Lemma 2]{IS5} on the operators $F_{\mid_{Im(DF)^{n}}} , D_{\mid_{Im(DF)^{n}}}  $ and $DF_{\mid_{Im(DF)^{n}}}  $ we deduce that $\ker DF_{\mid_{Im(DF)^{n}}} $ is finitely generated. Hence, $DF_{\mid_{Im(DF)^{n}}} $ is upper semi-$\mathcal{A}$-Fredholm by \cite[Lemma 12]{IS5}. Proceeding inductively we obtain that $DF_{\mid_{Im(DF)^{m}}} $ is upper semi-$\mathcal{A}$-Fredholm for all $m \geq n.$
	
	Suppose next that $F_{\mid_{Im F^{n}}} $ and $D_{\mid_{Im D^{n}}} $ are lower semi-$\mathcal{A}$-Fredholm. Then, by \cite[Lemma 12]{IS5}, $$Im F^{n}=Im F^{n+1} \oplus N  \text{ and } Im D^{n}=Im D^{n+1} \oplus N^{\prime} $$ 
	for some finitely generated Hilbert submodules $N$ and $N^{\prime}$. It follows that 
	$$Im D^{n}F^{n}=Im D^{n}F^{n+1} +D^{n}(N) \text{ and } Im F^{n}D^{n}=Im F^{n}D^{n+1}+F^{n}(N^{\prime}).$$  
	
	Since $Im F^{n+1} D^{n}=Im F(DF)^{n} $  and $Im D^{n+1}F^{n}=Im D(DF)^{n} $ are both closed by assumption, by \cite[Theorem 2.3.3]{MT}  we have that $Im F^{n+1}D^{n}$ and $Im D^{n+1}F^{n} $ are orthogonally complementable in 
	$$ImF^{n}D^{n}=ImD^{n}F^{n}=Im(DF)^{n},$$ 
	so 
	$$Im(DF)^{n}=Im F(DF)^{n} \oplus \tilde N \text{ and } Im(DF)^{n}=Im D(DF)^{n} \oplus \tilde N^{\prime} $$ 
	for some Hilbert submodules $\tilde N $ and $\tilde N^{\prime} .$ 
	Let $P$ and $P^{\prime} $ stand for the orthogonal projections onto $\tilde N $ and $\tilde N^{\prime},$ respectively. As $Im F^{n+1} D^{n}=Im D^{n}F^{n+1}$ and  $Im D^{n+1}F^{n}=Im F^{n}D^{n+1},$ it follows that $\tilde N = PD^{n}(N) $ and $\tilde N^{\prime} = P^{\prime}F^{n}(N^{\prime}) ,$ hence $ \tilde N  $ and $\tilde N^{\prime} $ are finitely generated since $N$ and $N^{\prime} $ are so. Thus, the orthogonal complement of $Im F(DF)^{n}  $ and the orthogonal complement of $Im D(DF)^{n} $ in $Im (DF)^{n} $ are both finitely generated. By applying again the exact sequence from \cite[Lemma 2]{IS5} on the operators $F_{\mid_{Im(DF)^{n}}} , D_{\mid_{Im(DF)^{n}}} $ and $DF_{\mid_{Im(DF)^{n}}},$ we obtain by \cite[Lemma 12]{IS5} that $DF_{\mid_{Im(DF)^{n}}} $ is lower semi-$\mathcal{A}$-Fredholm. Proceeding inductively we obtain  that $DF_{\mid_{Im(DF)^{m}}}  $ is lower semi-$\mathcal{A}$-Fredholm for all $m \geq n.$
	
	The proof in the case when $F$ and $D$ are $\mathcal{A}$-$B$-Fredholm is similar, or more precisely, a combination of the previous proofs for the cases when $D$ and $F$ were upper or lower semi-$\mathcal{A}$-$B$-Fredholm. Moreover, by applying the exact sequence from  \cite[Lemma 2]{IS5} in this case, we can also deduce that 
	$$\text {\rm index }  DF= \text {\rm index }  D+\text {\rm index }  F.$$
\end{proof}

In Proposition \ref{DP P26} we have considered various compositions of closed range $C^{*}$-operators under the additional assumption that these compositions also have closed image. The natural question which arises is what are the necessary and sufficient conditions for a composition of two closed range $C^{*}$-operators to have closed image. To answer this question, we give first the following lemma.

\begin{lemma} \label{D L22}     
	Let $M$ and $N$ be two closed submodules of a Hilbert $ C^{*} $-module $ \widetilde{M} $ over a $ C^{*} $-algebra $ \mathcal{A}. $ Suppose that $M$ is orthogonally complementable in $ \widetilde{M} $ and that $M \cap N =\lbrace 0 \rbrace .$ Then $M+N$ is closed if and only if $P_{\mid_{N}} $ is bounded below, where $P$ denotes the orthogonal projection onto $M^{\perp}.$ 
\end{lemma}

\begin{proof}
	Suppose first that $P_{\mid_{N}}$ is bounded below and let $ \delta=m (P_{\mid_{N}}).$ Then $\delta > 0.$ As in the proof of \cite[Lemma 3.2]{IS1} we wish to argue that in this case, there exists a constant $C>0$ such that if $x \in M $ and $y \in N $ satisfy $ \parallel x+y \parallel \leq 1 ,$ then $\parallel x \parallel \leq C .$ Now, since $M$ is orthogonally complementable, given $y \in N,$ we may write $y$ as $y=y^{\prime}+y^{\prime \prime} ,$ where $y^{\prime} \in M, y^{\prime \prime} \in M^{\perp}.$ Observe that $\langle y,y \rangle=$ $\langle y^{\prime},y^{\prime}\rangle +\langle y^{\prime \prime},y^{\prime \prime} \rangle .$ By taking the supremum over all states on $\mathcal{A} $ we obtain that $\parallel y \parallel \geq \max \lbrace \parallel  y^{\prime} \parallel , \parallel y^{\prime \prime} \parallel \rbrace .$  Hence $\parallel y^{\prime \prime} \parallel=\parallel P_{\mid_{N}} (y) \parallel \geq \delta \parallel y \parallel \geq \delta \parallel y^{\prime} \parallel .$ Then, by the same arguments as in the proof of  \cite[Lemma 3.2]{IS1}, we obtain that if  $\parallel x+y \parallel \leq 1$ and $x \in M,$ then $\parallel x \parallel \leq 1+\dfrac{1}{\delta}=\dfrac{\delta+1}{\delta} .$ It follows that $M+N$ is closed. 
	
	Conversely, if $M+N$ is closed, then, by Lemma \ref{D L18}, $M+N=M \oplus M^{\prime},$ where $M^{\prime}=M^{\perp} \cap (M+N) .$ Hence $P(M+N)=M^{\prime} ,$ which is closed. However, $P(M+N)=P(N) .$ Moreover, since $M \cap N=\lbrace 0 \rbrace ,$  we have that  $P_{\mid_{N}} $ is injective. By the Banach open mapping theorem it follows that $P_{\mid_{N}} $ is an isomorphism onto $M^{\prime} ,$ hence $P_{\mid_{N}} $ is bounded below. 
\end{proof}

Finally we are ready to give the conditions that are both necessary and sufficient for a composition of two closed range operators to have closed image.

\begin{corollary} \label{06c 08}   
	Let $ \widetilde{M} $ be a Hilbert $ C^{*} $-module, $F,D \in B^{a}(\widetilde{M} ) $ and suppose that $Im F, Im D$ are closed. Then $Im DF$ is closed if and only if $\ker D \cap Im F $ is orthogonally complementable and $P_{\mid_{Im F \cap (\ker D \cap Im F)^{\perp}}} $ is bounded below, (or, equivalently,  $Q_{\mid_{\ker D \cap (\ker D \cap Im F)^{\perp}}} $ is bounded below), where $P$ and $Q$ denote the orthogonal projections onto $\ker D^{\perp}  $ and $Im F^{\perp} ,$ respectively.
\end{corollary}

\begin{proof}
	If $\ker D \cap Im F $ is orthogonally complementable, then from Lemma \ref{D L18} it follows that $$\ker D = (\ker D \cap Im F) \oplus (\ker D \cap ( \ker D \cap Im F)^{\perp})$$ and $$Im F =(\ker D \cap Im F) \oplus (Im F \cap (\ker D \cap Im F)^{\perp})  .$$  
	Hence $$\ker D + Im F = \ker D + (Im F \cap (\ker D \cap Im F)^{\perp})$$ 
	$$= Im F + (\ker D \cap (\ker D \cap Im F)^{\perp})  . $$   If in addition $P_{Im F \cap (\ker D \cap Im F)^{\perp}} $ or $Q_{\mid_{\ker D \cap (\ker D \cap Im F)^{\perp}}} $ is bounded below, from Lemma \ref{D L22} ( as both $\ker D$ and $Im F$ are orthogonally complementable by \cite[Theorem 2.3.3]{MT} ) we deduce that $\ker D + Im F$   is closed. Then, from \cite[Corollary 1]{N} it follows that $Im DF$ is closed. \\
	
	Conversely, if $Im DF$ is closed, then $D_{\mid_{Im F}} $ is an adjointable operator with closed image. Indeed, since $Im F$ is closed, by  \cite[Theorem 2.3.3]{MT} $Im F$ is orthogonally complementable, hence $D_{\mid_{Im F}} $ is adjointable. From  \cite[Theorem 2.3.3]{MT} it follows that $ \ker D_{\mid_{Im F}} $ is orthogonally complementable in $Im F.$  However, $ \ker D_{\mid_{Im F}}=\ker D \cap Im F  .$ Since $Im F$ is orthogonally complementable in $ \widetilde{M} $ and $ \ker D \cap Im F \subseteq Im F,$  we get that $\ker D \cap Im F$ is orthogonally complementable in $ \widetilde{M}. $ Moreover, $\ker D + Im F $ is closed by \cite[Corollary 1]{N} since $Im DF$ is closed. By the previous arguments we have that
	$$\ker D = (\ker D \cap Im F ) \oplus (\ker D \cap (\ker D \cap Im F)^{\perp}),$$ 
	$$ Im F = (\ker D \cap Im F) \oplus (Im F \cap (\ker D \cap Im F)^{\perp})   ,$$ 
	so we are then in the position to apply Lemma \ref{D L22} which gives us the implication in the opposite direction.
\end{proof}

\begin{remark} \label{06r 09}  
	If $H$ is a Hilbert space and $M, N$ are closed subspaces of $H$ such that $M \cap N = \lbrace 0 \rbrace ,$ it is not hard to see that if $P$ denotes the orthogonal projection onto $M^{\perp} ,$ then $ P_{\mid_{N}}$ is bounded below if and only if the Dixmier angle between $M$ and $N$ is positive. Thus, Corollary \ref{06c 08}  is a proper generalization of Bouldin's result in \cite{Bld}. Indeed, since $ H $ is a Hilbert space, for each $ y\in N $ we have that  $ \parallel y \parallel ^{2}= \parallel P_{\mid_{N}}y \parallel ^{2}+ \parallel (I-P_{\mid_{N}})y \parallel ^{2}   .$ 
	So, $  \parallel (I-P_{\mid_{N}})y \parallel =\sqrt{ \parallel y \parallel ^{2}- \parallel P_{\mid_{N}}y \parallel ^{2} } $ for every $ y\in N,$ in particular $  \parallel (I-P_{\mid_{N}})y \parallel =\sqrt{1- \parallel P_{\mid_{N}}y \parallel ^{2}} $ for every $ y\in N $ with $  \parallel y \parallel =1. $ Next, for each $ y\in N, $ we have 
	$sup \lbrace \mid \langle x,y \rangle \mid \vert \text{ } x\in M \text{ and } \parallel x \parallel \leq 1 \rbrace=  \parallel(I-P_{\mid_{N}})y \parallel. $
	
	This is because $ \mid \langle x,y \rangle \mid=\mid \langle x,(I-P_{\mid_{N}})y \rangle  \mid \leq  \parallel (I-P_{\mid_{N}})y \parallel  $ when $ x\in M $ with  $ \parallel x \parallel \leq1, $ and, on the other hand, $ \vert \langle y^{\prime},y \rangle \vert = \parallel (I-P_{\mid_{N}})(y) \parallel , $ where  
	
	\begin{eqnarray*}
		y^{\prime}=
		\left\{
		\begin{array}{cc}
			\frac{(I-P_{\mid_{N}})y}{\parallel (I-P_{\mid_{N}})y \parallel}\, & \text{if } (I-P_{\mid_{N}})y \neq 0, \\ \\
			0 & \text{if } (I-P_{\mid_{N}})y = 0.\\ \\
		\end{array}
		\right.
	\end{eqnarray*}
	
	Thus, $  \parallel y^{\prime} \parallel \leq 1 $ and $ y^{\prime}\in M.  $ Therefore,  $$\text{sup }\lbrace  {\mid \langle x,y \rangle \mid \vert \text{ } x\in M,   \parallel y^{\prime} \parallel  \leq 1}\rbrace = \parallel (I-P_{\mid_{N}})y \parallel  $$ 
	for every $y\in N.$ Combining all this together, we deduce that 
	$$ c_{0}(M,N)= \text {sup } \lbrace \sqrt{1-\parallel P_{\mid_{N}} y \parallel^{2}   } \mid y \in N, \parallel y \parallel =1 \rbrace ,$$ hence $c_{0}(M,N) <1  $ if and only if $P_{\mid_{N}}    $ is bounced below.
	
\end{remark}

Now we give some examples of $\mathcal{A}$-Fredholm operators with non-closed image.

\begin{example} \label{06e 17}     
	Let $ \mathcal{A}=L^{\infty}((0,1), \mu) $ and consider the operator $F: \mathcal{A} \rightarrow \mathcal{A}  $ given by $F(f)=f \cdot id$ (where $id(x)=x  $ for all $x \in (0,1)  $). Then $F$ is an $\mathcal{A}$-linear, bounded operator on $\mathcal{A}$ and, since $\mathcal{A}$ is finitely generated considered as Hilbert $\mathcal{A}$-module over itself, it follows that $F$ is $\mathcal{A}$-Fredholm.  However, $Im F$ is not closed. Indeed, $ \parallel F (\mathcal{X}_{(0,\frac{1}{n})}) \parallel_{\infty}=\frac{1}{n} $ for all $n$ whereas $\parallel (\mathcal{X}_{(0,\frac{1}{n})}) \parallel_{\infty}=1  $ for all $n,$ so $F$ is not bounded below.\\
	Consider now the operator $\tilde{F} \in  B^{a}(H_{\mathcal{A}})  $ given by $\tilde{F} =Q+JFP,$ where $Q$ denotes the orthogonal projection onto $L_{1}^{\perp},P=I-Q  $ and $J(\alpha)=(\alpha,0,0,0, \dots)  $ for all $ \alpha \in \mathcal{A} .$ Then it is easy to see that $\tilde{F} \in  \mathcal{M}\Phi (H_{\mathcal{A}})  $ and $Im \tilde{F}  $ is not closed. 
\end{example}

\begin{example} \label{06e 19}    
	Let $\mathcal{A}= B(H) $ where $H$ is a Hilbert space. Choose an $S \in B(H)  $ such that $Im S$ is not closed. Then $S$ is not bounded below, so there exists a sequence of unit vectors $ \lbrace x_{n } \rbrace_{n \in \mathbb{N} }  $ in $H$ such that $\parallel Sx_{n}\parallel \rightarrow 0$ as $ n \rightarrow \infty .$ Choose an $x \in H $ such that $\parallel x \parallel=1 $ and define the operators $ B_{n} \in B(H)$ to be given as $B_{n}x =x_{n}$ and ${B_{n}}_{\mid Span \lbrace x \rbrace^{\perp} } =0$ for all $n.$ Then we have that $\parallel B_{n} \parallel=\parallel B_{n}x \parallel=\parallel x_{n} \parallel = 1 $ for all $n.$ However, since $S{B_{n}}_{\mid Span \lbrace x \rbrace^{\perp} }=0 $ for all $n$ and $\parallel x \parallel = 1 ,$ it follows that    $\parallel SB_{n}\parallel= \parallel SB_{n}x \parallel=\parallel Sx_{n}\parallel $ for all $n.$ Thus, $\parallel SB_{n}\parallel \rightarrow 0 $ as $n \rightarrow \infty .$ If we consider the operator $F: \mathcal{A} \rightarrow \mathcal{A}  $ given by $F(T) = ST$ for all $T \in B(H) ,$ then $F$ is an $\mathcal{A}$-linear, bounded operator on $\mathcal{A}$ (when $\mathcal{A}$ is viewed as a Hilbert $\mathcal{A}$-module over itself), but $Im F$ is not closed. This also follows from \cite[Theorem 7]{JS}. Using the operator $F,$ it is easy to construct an operator $\tilde{F} \in  \mathcal{M}\Phi (H_{\mathcal{A}})  $ in the same way as in the previous example such that $Im \tilde{F}  $ is not closed.
\end{example}

Notice that if $ S\in B(H) $ is such that $ ImS $ is closed , but $ ImS^{2} $ is not closed, then $ Im\tilde{F} $ will be closed , whereas $ Im \tilde{F}^{2} $ will \underline{not} be closed. Now we will give another example of an $ \mathcal{A}$-Fredholm operator $F$ with the property that $ ImF $ is closed, but $ ImF^{2} $ is not closed.

\begin{example} \label{06e 18}    
	Let $H$ be an infinite-dimensional Hilbert space, $M$ and $N$ be closed, infinite-dimensional subspaces of $H$  such that  $M+N$ is not closed. Denote by $p$ and $q$ the orthogonal projections onto $M$ and $N,$ respectively. If we let $\mathcal{A}=B(H) ,$  then $\tilde{M}=Span_{\mathcal{A}} \lbrace (p,0,0,0,\dots) \rbrace $ and $\tilde{N}=Span_{\mathcal{A}} \lbrace (q,0,0,0,\dots) \rbrace $ are finitely generated Hilbert submodules of $H_{\mathcal{A}} .$ Moreover, $\tilde{M}+\tilde{N} $ is not closed. Indeed, since $M+N$ is not closed, there exists a sequence $\lbrace x_{n}+y_{n} \rbrace $ in $H$ such that $ x_{n} \in M, y_{n} \in N $ for all $n$ and $x_{n}+y_{n} \rightarrow z  $ for some $z \in H \setminus (M+N) .$ Choose an $x \in H $ such that $\parallel x \parallel=1 $ and let, for each $n,$ $T_{n} $ and $S_{n} $ be the operators in $B(H)$  defined by $T_{n}x=x_{n},$ $S_{n}x=y_{n}$ and ${T_{n}}_{\mid Span \lbrace x \rbrace^{\perp} }={S_{n}}_{\mid Span \lbrace x \rbrace^{\perp} }=0 .$ Since $x_{n} \in M $ and $y_{n} \in N $ for all $n,$ it follows that $T_{n} \in p\mathcal{A}  $ and $S_{n} \in q\mathcal{A} $ for all $n.$ Moreover, $\parallel S_{n}+T_{n}-S_{m}-T_{m} \parallel =\parallel (S_{n}+T_{n}-S_{m}-T_{m})x \parallel  $ for all $m, n.$ Since $(S_{n}+T_{n})x=x_{n}+y_{n} $ for all $n,$ it follows that $\lbrace S_{n}+T_{n} \rbrace_{n}  $ is a Cauchy sequence in $B(H),$ hence $S_{n}+T_{n} \rightarrow T $ for some $T \in  B(H)  .$ Then $x_{n}+y_{n}=S_{n}x+T_{n}x \rightarrow Tx=z$ as $n \rightarrow \infty.$ Now, $S_{n}+T_{n} \in p\mathcal{A}+q\mathcal{A} $ for all $n.$ If also $T \in p\mathcal{A}+q\mathcal{A} ,$ then $Tx \in M+N .$ However, then $z \in M+N ,$ which is a contradiction. Thus, $T \notin  p\mathcal{A}+q\mathcal{A} ,$ so $p\mathcal{A}+q\mathcal{A} $ is not closed in $\mathcal{A}.$ It follows easily that $\tilde{M}+\tilde{N} $ is not closed. Also, $(L_{1}^{\perp} \oplus \tilde{M})+\tilde{N} $ is not closed. Since $\tilde{N} $ is finitely generated, by the Dupre-Filmore  Theorem \cite[Theorem 1.4.5]{MT} we have that $\tilde{N}^{\perp} \cong H_{\mathcal{A}} .$ Moreover, $L_{1}^{\perp} \oplus \tilde{M} \cong H_{\mathcal{A}} ,$ hence $L_{1}^{\perp} \oplus \tilde{M} \cong \tilde{N}^{\perp} .$ Let $U:\tilde{N}^{\perp} \rightarrow L_{1}^{\perp} \oplus \tilde M $ be an isomorphism, set $F=JUP ,$ where $P$ is the orthogonal projection onto $\tilde{N}^{\perp} $ and $J$ is the inclusion from $L_{1}^{\perp} \oplus \tilde{M} $ into $H_{\mathcal{A}}. $ Then $\ker F =\tilde{N} $ and $Im F=L_{1}^{\perp} \oplus \tilde{M} ,$ so $F$ is $\mathcal{A}$-Fredholm. Now, since $ImF+\ker F  $ is not closed, it follows from  \cite[Corollary 1]{N}  that $ Im F^{2} $ is not closed.
\end{example}

These examples show that semi-$ \mathcal{A} $-Fredholm operators may behave differently from classical semi-Fredholm operators on Hilbert spaces. Indeed, classical semi-Fredholm operators always have closed image and are therefore regular operators on Hilbert spaces.

\section{Drazin invertible $C^{*}$-operators and \newline $C^{*}$-Browder  operators}

In this section, we consider Drazin invertible $C^{*}$-operators and $C^{*}$-Browder operators as a generalization of Drazin invertible and Browder operators on Hilbert spaces.

Let $M$ be a Hilbert $C^{*}$-module. We recall that an operator $F \in B^{a}(M) $ is said to be Drazin invertible if there exists some $p $ such that $Im F^{k} $ is closed for all $k \geq p $ and $Im F^{k}=Im F^{p}, \text{ } \ker F^{k}=\ker F^{p} $ for all $k \geq p ,$ that is $asc F = dsc F = p.$ In this case $F$ has the matrix 
$\left\lbrack
\begin{array}{cc}
F_{1} & 0 \\
0 &F_{4} \\
\end{array}
\right \rbrack,
$
with respect to the decomposition $M=Im F^{p} \tilde{ \oplus} \ker F^{p} $ where $F_{1} $ is an isomorphism.

First, we give the following lemma.

\begin{lemma} \label{RA L08}
	Let $M$ be a Hilbert $C^{*}-$module and $F \in B^{a}(M) .$ Then $F$ is Drazin invertible if and only if $F^{*}$ is Drazin invertible.
\end{lemma}

\begin{proof}
	If $F$ is Drazin invertible, then there exists some $p $ such that $Im F^{k}=Im F^{p} $ and $\ker F^{k} = \ker F^{p} $ for all $k \geq p $ and moreover, $Im F^{k} $ is closed for all $k \geq p .$ By the proof of \cite[Theorem 2.3.3]{MT} we get that $Im F^{*k} $ is also closed for all $ k \geq p.$ Since we have by \cite[Theorem 2.3.3]{MT} that $M=Im F^{k} \oplus \ker F^{*k} = Im F^{*k} \oplus \ker F^{k} $ for all $k \geq p ,$ it follows that $\ker F^{*k} = \ker F^{*p} $ and $Im F^{*k} = Im F^{*p} $ for all $k \geq p $ as $ Im F^{k} = Im F^{p}$ and $ \ker F^{k} = \ker F^{p} $ for all $k \geq p .$ Hence $F^{*}$ is Drazin invertible. By applying the same argument on $F = (F^{*})^{*} ,$ we deduce that $F$ is Drazin invertible if $F^{*}$ is Drazin invertible. 
\end{proof}

Now we present the main result in this section.

\begin{proposition} \label{RA L01}
	Let $ F,D \in B^{a}(M) $ such that $FD=DF.$ Assume that $FD$ is Drazin invertible and let $p \in \mathbb{N} $ be such that $Im (FD)^{n}=Im (FD)^{p},$ $\ker (FD)^{n}= \ker (FD)^{p} $ for all $n \geq p .$ Then $F$ is Drazin invertible if and only if there exist some $s,t,k,k^{\prime} \in \mathbb{N}$  such that $p \leq k \leq k^{\prime} ,$ $Im F^{k}$ and $Im F^{k+s}$ are closed, $Im F^{k} \cap \ker D^{p} = Im F^{k+s} \cap \ker D^{p} $ and $Im F^{*k^{\prime}} \cap \ker D^{*p}=Im F^{*k^{\prime}+t} \cap \ker D^{*p} .$ 
\end{proposition}

\begin{proof}
	We observe first that if $ k \geq p ,$ then $Im D^{k}F^{k} \subseteq Im D^{p}F^{k} \subseteq Im D^{p}F^{p} .$ Since $Im D^{k}F^{k} = Im D^{p}F^{p} ,$ we get $Im D^{p}F^{k} = Im D^{p}F^{p} .$ Similarly, we have $Im D^{p} F^{k+s} = Im D^{p} F^{p} ,$ hence $ Im D^{p} F^{k}=Im D^{p} F^{k+s}  .$ Since $Im  F^{k+s} $ is closed by assumption, by \cite[Theorem 2.3.3]{MT} $Im F^{k+s}$ is orthogonally complementable in $M.$ Hence $ D_{\mid_{Im F^{k+s}}}^{p}$ is adjointable.  By applying now \cite[Theorem 2.3.3]{MT} on the operator $D_{\mid_{Im F^{k+s}}}^{p} ,$ we deduce that $Im F^{k+s} \cap \ker D^{p} (= \ker D_{\mid_{Im F^{k+s}}}^{p}) $ is orthogonally complementable in $Im F^{k+s}.$ Thus $Im F^{k+s}=( \ker D^{p} \cap Im F^{k+s} ) \oplus Y$ for some closed submodule $Y.$ Now, since $Im F^{k+s}$ is orthogonally complementable in $M,$ by \cite[Lemma 2.6]{IS3} we deduce that $Im F^{k+s} \oplus Z = Im F^{k} $ for some closed submodule $Z.$ Hence we get that $ Im F^{k}=(\ker D^{p} \cap Im F^{k+s}) \oplus Y \oplus Z.$ Moreover, $Im D^{p} F^{k+s}=D^{p} (Y)= Im D^{p} F^{k}=D^{p}(Y) \tilde{\oplus}D^{p}(Z) .$ It follows that given $z \in Z ,$ there exists some $y \in Y $ such that $D^{p}(y)=D^{p}(z) ,$ hence $y-z \in \ker D^{p} \cap Im F^{k} .$ If $\ker D^{p} \cap Im F^{k}=\ker D^{p} \cap Im F^{k+s} ,$ we get that $y-z \in \ker D^{p} \cap Im F^{k+s} .$ In particular, $y-z \in Im F^{k+s} .$ Since $y \in Im F^{k+s} ,$ we must have that $z \in Im F^{k+s} $ also. As $ Im F^{k+s} \cap Z = \lbrace 0 \rbrace,$ we get that $z=0 .$ Hence $Z=\lbrace 0 \rbrace $ because $z \in Z $ was chosen arbitrary. Thus we obtain that $Im F^{k+s} = Im F^{k} .$ It follows that $Im F^{n}=Im F^{k} $ for all $ n \geq k.$ In particular, $Im F^{n} $ is then closed for all $n \geq k .$ By the proof of \cite[Theorem 2.3.3]{MT}, it follows that $Im F^{*n} $ is also closed for all $ n \geq k .$ Since $k^{\prime} \geq k $ we get that $ Im F^{*k^{\prime}}$ and $Im F^{*k^{\prime}+t} $ are closed. Moreover, from the proof of the  Lemma \ref{RA L08}, we have that $(FD)^{*} $ is Drazin invertible since $FD$ is so, and 
	$Im (FD)^{*n}=\ker (FD)^{n^{\perp}} =\ker (FD)^{p^{\perp}}= Im (FD)^{*p}$ for all $n \geq p .$ Finally, $F^{*}$ an $D^{*}$ mutually commute since $FD=DF.$ Hence, we can apply the preceding arguments on the operators $F^{*}$ and $D^{*}$ instead of the operators $F$ an $D,$ respectively, in order to deduce that $ Im F^{*n} = Im F^{*k^{\prime}}$ for all $n \geq k^{\prime} .$ This gives $\ker F^{n}= (Im F^{*n} )^{\perp}= (Im F^{*k^{\prime}})^{\perp} = \ker F^{k^{\prime}} $ for all $n \geq k^{\prime} .$ Thus, $F$ is Drazin invertible. 
	
	Conversely, if $F$ is Drazin invertible, then there exists some $k$  such that $Im F^{k} $ is closed and such that $Im F^{n}=Im F^{k} $ and $\ker F^{n}=\ker F^{k} $ for all $n \geq k .$ As observed earlier, this implies also that $Im F^{*n}=Im F^{*k} ,$ which is closed, for all $n \geq k .$ 
\end{proof}

\begin{corollary}
	Let $H$ be a Hilbert space and $F,D \in B(H) $ such that $FD=DF.$ If $DF$ is Drazin invertible and $D$ is Fredholm, then $F$ is Drazin invertible if $Im F^{k} $ is closed for all $k \geq p $ where $asc (FD)= dsc(FD)=p .$ 
\end{corollary}

\begin{proof}
	Since $D$ is Fredholm, we have that $\ker D^{p} $ and $\ker D^{*p} $ are finite dimensional. Hence there exist some $k_{0}^{\prime} \geq k_{0} \geq p $ such that 
	$Im F^{k} \cap \ker D^{p} = Im F^{k_{0}} \cap \ker D^{p}  $ for all $k \geq k_{0} $ and $Im F^{*k} \cap \ker D^{*p}=Im F^{*k_{0}^{\prime}} \cap \ker D^{*p} $ for all $k \geq k_{0}^{\prime} .$
\end{proof}

\begin{remark}
	Actually, it suffices to assume that there exists a strictly increasing sequence $ \lbrace n_{k} \rbrace_{k} \subseteq \mathbb{N} $ such that $Im F^{n_{k}} $ is closed for all $k.$
\end{remark}

Next we give examples of two mutually commuting $C^{*}$-operators whose composition is Drazin invertible whereas they are not Drazin invertible.

\begin{example}
	Let $H$ be a separable infinite dimensional Hilbert space and $S$ be a unilateral shift operator on $H.$ Then $ \ker S^{*} \subsetneqq \ker S^{*^{2}} \subsetneqq \ker S^{*^{3}} \subsetneqq \dots .$
	In addition, $Im S^{k} $ and $Im S^{*^{k}} $ are closed for all $k.$ 
	
	Consider $ H_{\mathcal{A}} $ where $\mathcal{A}=B(H) .$ Let $L_{S}$ and $L_{S^{*}}$  be the left multipliers by $S$ and $S^{*} ,$ respectively, and put $F$ to be the operator on $H_{\mathcal{A}} $ with matrix 
	$\left\lbrack
	\begin{array}{ll}
	F_{1} & 0 \\
	0 & L_{S} \\
	\end{array}
	\right \rbrack
	$
	with respect to the decomposition $L_{1}(\mathcal{A})^{\perp} \oplus L_{1}(\mathcal{A}) ,$ where $ F_{1}$ is an isomorphism. If $P $ stands for the orthogonal projection onto $L_{1}(\mathcal{A})^{\perp} ,$ then $FP$ is obviously Drazin invertible. However, $Im F \supsetneqq Im F^{2} \supsetneqq Im F^{3}  \supsetneqq \dots .$ Similarly, since $F^{*}$ has the matrix 
	$\left\lbrack
	\begin{array}{ll}
	F_{1}^{*} & 0 \\
	0 & L_{S^{*}} \\
	\end{array}
	\right \rbrack
	$ 
	with respect to the same decomposition, (and $F_{1}^{*} $ is an isomorphism as $F_{1} $ is so), it follows easily that $F^{*}P $ is also Drazin invertible, however, $\ker F^{*} \subsetneqq \ker F^{*^{2}} \subsetneqq \ker F^{*^{3}} \subsetneqq \dots  .$ On the other hand, $FP=PF $ and $F^{*}P=PF^{*},$ however, we also have that $(Im F \cap \ker P) \supsetneqq (Im F^{2} \cap \ker P) \supsetneqq (Im F^{3} \cap \ker P) \supsetneqq \dots $ as $Im F^{k} \cap \ker P = \lbrace (G,0,0, \dots) \mid G \in Im L_{S^{k}} $ for all $k \in \mathbb{N} \rbrace .$ Next, let $T \in B(H)=\mathcal{A} $ such that $Im T^{k} $ is not closed for all $k \in \mathbb{N} .$ Then, if $F$ has the matrix 
	$\left\lbrack
	\begin{array}{ll}
	F_{1} & 0 \\
	0 & L_{T} \\
	\end{array}
	\right \rbrack
	$ 
	with respect to the decomposition $L_{1}(\mathcal{A})^{\perp} \oplus L_{1}(\mathcal{A}) ,$ it is easily seen that $ FP=PF$ and $FP$ is Drazin invertible if $F_{1} $ is an isomorphism, however, $Im F^{k} $ is not closed for all $k,$ hence $F$ is not Drazin invertible.
\end{example}

\begin{example}
	Let $H, S$ and $T$ be as in the previous example and $H_{1} $ be another Hilbert space. Consider the Hilbert space $H_{2}:=H_{1} \oplus H.$ If $F$ has the matrix 
	$\left\lbrack
	\begin{array}{ll}
	F_{1} & 0 \\
	0 & S \\
	\end{array}
	\right \rbrack
	$ 
	with respect to the decomposition $H_{1} \oplus H $ and $P$ denotes now orthogonal projection onto $H_{1} ,$ then $FP=PF ,$ $FP$ is Drazin invertible if $F_{1} $ is an isomorphism, however, $ Im F \supsetneqq Im F^{2} \supsetneqq Im F^{3} \supsetneqq \dots .$ Also, $F^{*}P $ is Drazin invertible and $F^{*}P=PF^{*} ,$ however, $\ker F^{*} \subsetneqq \ker F^{*^{2}} \subsetneqq \ker F^{*^{3}} \subsetneqq \dots   .$ Finally, if $D$ has the matrix 
	$\left\lbrack
	\begin{array}{ll}
	D_{1} & 0 \\
	0 & T \\
	\end{array}
	\right \rbrack
	$  
	with respect to the same decomposition, then $ DP=PD$ and $DP$ is Drazin invertible if $D_{1} $ is an isomorphism, however, $Im D^{k}$ is \underline{not} closed for all $k.$ 
\end{example}

\begin{definition} \label{R D09}
	Let $F \in B^{a} (H_{\mathcal{A}}) .$ We say that $F$ is $\mathcal{A} -$Browder if there exists an $\mathcal{A}-$Fredholm decomposition for $F$ of the form $H_{\mathcal{A}} = M \tilde{\oplus} N \stackrel{F}{\longrightarrow} M \tilde{\oplus} N = H_{\mathcal{A}} .$
\end{definition}

\begin{lemma} \label{r11 l0.5.10}
	Let $F,D \in B^{a} (H_{\mathcal{A}})$ such that $FD=DF.$ If $DF$ is Drazin invertible and $\mathcal{A}-$Fredholm, then $F$ and $D$ are $\mathcal{A}-$Browder. 
\end{lemma}

\begin{proof}
	If $DF$ is Drazin invertible, then there exists some $p \in \mathbb{N}$ such that $DF$ has the matrix 
	$\left\lbrack
	\begin{array}{ll}
	(DF)_{1} & 0 \\
	0 & (DF)_{4} \\
	\end{array}
	\right \rbrack
	$  
	with respect to the decomposition 
	$$H_{\mathcal{A}} = Im (DF)^{p} \tilde{\oplus} \ker (DF)^{p} \stackrel{DF}{\longrightarrow} Im (DF)^{p} \tilde{\oplus} \ker (DF)^{p} = H_{\mathcal{A}} $$
	
	where $(DF)_{1}$ is an isomorphism. Since $DF$ maps $Im (DF)^{p}$ isomorphically  onto itself, it is not hard to see that $F_{\mid_{Im (DF)^{p}}} $ is an isomorphism onto $Im F(DF)^{p}$ and $D_{\mid_{Im F(DF)^{p}}}$ is an isomorphism onto $Im (DF)^{p} .$ However, as we have observed earlier, $Im F^{p+1}D^{p+1} = Im F^{p+1}D^{p} = Im F^{p}D^{p+1} ,$ hence, $F$ and $D$ map  $Im (FD)^{p} $ isomorphically onto itself. Next, since 
	$$\ker F^{p}D^{p} \subseteq \ker F^{p+1}D^{p} \subseteq \ker DF^{p+1}D^{p} = \ker D^{p+1}F^{p+1}  ,$$ 
	we must have that 
	$$\ker F^{p}D^{p}= \ker F^{p+1}D^{p} = \ker F^{p+1}D^{p+1} .$$ 
	Thus, $F$ and $D$ have the matrices 
	$\left\lbrack
	\begin{array}{ll}
	F_{1} & 0 \\
	0 & F_{4} \\
	\end{array}
	\right \rbrack
	$ 
	and 
	$\left\lbrack
	\begin{array}{ll}
	D_{1} & 0 \\
	0 & D_{4} \\
	\end{array}
	\right \rbrack
	,$ 
	respectively, with respect to the decomposition $H_{\mathcal{A}}= Im (FD)^{p} \tilde{\oplus} \ker (FD)^{p} ,$ where $ F_{1} $ and $ D_{1} $ are isomorphisms.\\
	Finally, by \cite[Lemma 2.7.11]{MT}, $ (FD)^{p}$ is $\mathcal{A}-$Fredholm since $FD$ is $\mathcal{A}-$Fredholm. Since $ Im (FD)^{p}$ is closed, from \cite[Lemma 12]{IS5} it follows that $\ker (FD)^{p} $ is finitely generated.	
\end{proof}

\begin{remark}    
	Recall that if $H$ is a Hilbert space and $F \in B(H),$ then $F$ is a Browder operator on $H$ if $F$ is Fredholm and Drazin invertible. Since finitely generated Hilbert subspaces are simply finite dimensional subspaces, it is not hard to see that, in the case of Hilbert spaces, Definition \ref{R D09} correspond to the definition of classical Browder operators.
\end{remark}

\begin{corollary} \cite[Theorem 2.8.2]{ZZRD}
	Let $H$ be a Hilbert space and $F,D \in B(H) $such that $FD=DF$. If $DF$ is Browder, then $F$ and $D$ are Browder.
\end{corollary}

\begin{proof}
	If $FD=DF$ is Browder, then it is Fredholm and Drazin invertible, hence the statement follows from Lemma \ref{r11 l0.5.10}. 
\end{proof}


\begin{thebibliography}{}


\bibitem{BS}[BS] M. Berkani and M. Sarih, \textit{On semi B-Fredholm operators}, Glasgow Mathematical Journal. Cambridge Univ. Press, Cambridge. ISSN 0017-0895, Volume 43, Issue 3. May 2001 , pp. 457-465, DOI: https://doi.org/10.1017/S0017089501030075

\bibitem{BM}[BM] M. Berkani, \textit{Index of B-Fredholm operators and Generalizations of a Weyl Theorem}, Proceedings of the American Mathematical Society. Amer. Math. Soc., Providence, RI. ISSN 0002-9939 , Volume 130, Number \textbf{6}, Pages 1717-1723, S 0002-9939(01)06291-8, Article electronically published on October 17, 2001

\bibitem{Bld} [Bld] R. Bouldin, \textit{The product of operators with closed range}, The Tohoku Mathematical Journal. Second Series. Tohoku Univ., Sendai. ISSN 0040-8735. , \textbf{25} (1973), 359-363.

\bibitem{DDj2}[DDj2] D. S. \DJ{}or\dj{}evi\'{c}, \textit{On generalized Weyl operators}, Proceedings of the American Mathematical Society. Amer. Math. Soc., Providence, RI. ISSN 0002-9939, Volume 130, Number \textbf{1}, Pages 81 ll4, s 002-9939(01)0608r-6, April 26,2001


\bibitem{H} [H] R.E. Harte, \textit{The ghost of an index theorem}, Proceedings of the American Mathematical Society. Amer. Math. Soc., Providence, RI. ISSN 0002-9939, \textbf{106} (1989). 1031-1033. MR 92j:47029

\bibitem{IS1}[IS1] S. Ivkovi\'{c} , \textit{Semi-Fredholm theory on Hilbert C*-modules,} Banach Journal of Mathematical Analysis. Tusi Math. Res. Group (TMRG), Mashhad. ISSN 1735-8787 , Vol. \textbf{13} (2019) no. 4 2019, 989-1016 doi:10.1215/17358787-2019-0022. https://projecteuclid.org/euclid.bjma/1570608171

\bibitem{IS3} [IS3] S. Ivkovi\'{c}, \textit{On operators with closed range and semi-Fredholm operators over W*-algebras}, Russian Journal of Mathematical Physics. MAIK Nauka/Interperiod. Publ., Moscow. ISSN 1061-9208, 27, 4860 (2020)  http://link.springer.com/article/10.1134/S1061920820010057

\bibitem{IS5}[IS5] S. Ivkovi\'{c}, \textit{On various generalizations of semi-A-Fredholm operators}, Complex Analysis and Operator Theory. Birkh\"auser, Basel. ISSN 1661-8254. 14, 41 (2020). https://doi.org/10.1007/s11785-020-00995-3 

\bibitem{IS9} [IS9] S. Ivkovi\'{c}, \textit{Semi-Fredholm operators on Hilbert C*-modules}, Doctoral dissertation, Faculty of Mathematics, University of Belgrade  (2022), 

\bibitem{IS11} [IS11] S. Ivkovi\'{c}, \textit{On Drazin invertible C*-operators and generalized C*-Weyl operators}, Ann. Funct. Anal. 14, 36 (2023). https://doi.org/10.1007/s43034-023-00258-0

\bibitem{JS} [JS] P. S. Johnson,  \textit{  Multiplication Operators with Closed Range in Operator Algebras}, Journal of Analysis and Number Theory, ISSN 2375-2785  (print), ISSN 2375-2807 (online) , No. \textbf{1}, 1-5 (2013) 

\bibitem{KM}[KM] M. Karubi, \textit{K-theory. An introduction}, Grundlehren der Mathematischen Wissenschaften. [Fundamental Principles of Mathematical Sciences] Springer, Heidelberg. ISSN 0072-7830, vol 226, Springer-Verlag, Berlin - Heidelberg - New York, 1978


\bibitem{KY} [KY] K. W. Yang, \textit{The generalized Fredholm operators},  Transactions of the American Mathematical Society, Providence, RI. ISSN 0002-9947,
Vol. \textbf{216} (Feb., 1976), pp. 313-326

\bibitem{MT}[MT] V. M. Manuilov, E. V. Troitsky, \textit{Hilbert C*-modules}, In: Translations of Mathematical Monographs. 226, Journal of the American Mathematical Society. Amer. Math. Soc., Providence, RI. ISSN 0894-0347, Providence, RI, 2005.

\bibitem{MF}[MF] A. S. Mishchenko, A.T. Fomenko, \textit{The index of eliptic operators over C*-algebras}, Izvestiya Akademii Nauk SSSR. Seriya Matematicheskaya \textbf{43} (1979), 831--859; English transl., Math. USSR-Izv.\textbf{15} (1980) 87--112.

\bibitem{N} [N] G. Nikaido   , \textit{Remarks on the Lower Bound of a linear Operator},  Japan Academy. Proceedings. Series A. Mathematical Sciences. Japan Acad., Tokyo. ISSN 0386-2194, Volume \textbf{56}, Number 7 (1980), 321-323.

\bibitem{P}[P] W. L. Paschke, \textit{Inner product modules over $B^{*}$-algebras}, Transactions of the American Mathematical Society. Amer. Math. Soc., Providence, RI. ISSN 0002-9947, \textbf{182} (1973), 443--468.

\bibitem{RW}[RW] W. Rudin , \textit{Functional Analysis - 2nd ed} , International series in pure and applied mathematics, ISBN 0-07-054236-8

\bibitem{W}[W] N. E. Wegge Olsen, \textit{ K-theory and C*-algebras}, Oxford Science Publications, The Clarendon Press Oxford University Press, ISSN 0017 3835, New York,  1993.	

\bibitem{ZZRD}[ZZRD] S. \v{Z}ivkovi\'{c} Zlatanovi\'{c}, V. Rako\v{c}evi\'{c}, D.S. \DJ{}or\dj{}evi\'{c}, \textit{Fredholm theory}, University of Ni\v{s} Faculty of Sciences and Mathematics, Ni\v{s},to appear (2022).



\end{thebibliography}
\end{document}